\documentclass[letterpaper,11pt]{amsart}
\usepackage[margin=1.8cm]{geometry}
\usepackage{amsmath,bm}
\usepackage{calc, comment}
\usepackage[dvipsnames]{xcolor}
\usepackage{accents}
\usepackage[all,cmtip]{xy}                      %

\usepackage{tikz}
\usepackage{soul}

\CompileMatrices                            

\UseTips                                    

\input xypic
\usepackage[bookmarks=true]{hyperref}       

\usepackage{amssymb,latexsym,amsmath,amscd}
\usepackage{xspace}
\usepackage{color}
\usepackage{graphicx}
\usepackage{dsfont}
\usepackage{enumitem}
\usepackage{hyperref}
\hypersetup{
	colorlinks,
	linkcolor={MidnightBlue},
	citecolor={Green}
}
\usepackage[nameinlink,noabbrev]{cleveref}
\usepackage[normalem]{ulem}
\usepackage{centernot} 

\reversemarginpar

\usepackage{listings}
\lstdefinelanguage{code}{
basicstyle=\small\ttfamily,
alsoletter=",
classoffset=1,
keywords={gb, eliminate, saturate, diff, degree, flatten, apply, tensor, product},
keywordstyle={\color{teal}},
classoffset=2,
morekeywords={from, to, list, terms, toList, entries, for, end, if, return},
keywordstyle={\color{blue}},
classoffset=3,
morekeywords={QQ},
classoffset=4,
morekeywords={ideal, matrix, gens},
keywordstyle={\color{teal}},
xleftmargin=1.5cm,
xrightmargin=1em,
columns=fullflexible,
keepspaces=true,
stepnumber=1,
numbers=none,
captionpos=b,
showspaces=false,
frame=none
}
\theoremstyle{plain}
\newtheorem{theorem}{Theorem}[section]
\newtheorem*{theorem*}{Theorem}
\newtheorem{proposition}[theorem]{Proposition}

\newtheorem{lemma}[theorem]{Lemma}

\theoremstyle{definition}

\newtheorem{remark}[theorem]{Remark}

\newcommand{\enm}[1]{\ensuremath{#1}}          %
\newcommand{\tv}{\mathcal{T}}
\newcommand{\PP}{\enm{\mathbb{P}}}

\newcommand{\Sing}{\enm{\mathrm{Sing}}}

\newcommand{\C}{\mathbb{C}}
\newcommand{\ed}{\mathrm{exp dim}}

\title{Some Taylor varieties with null Hessian} 
\date{\today}

\author[WINE-5 Taylor Varieties group]{Thais Gomes Ribeiro \and Elena Guardo \and Manuela Muzika Dizdarevi\'c \and Maryam Nowroozi \and Pierpaola Santarsiero \and Paola Supino }

\newcommand{\Addresses}{{
  \bigskip
  \footnotesize

\textsc{Thais Gomes Ribeiro, School of Mathematics, University of Birmingham,  Ring Road N B152TT, UK } \par\nopagebreak
\textit{E-mail address}: \email{txg306@student.bham.ac.uk}

\medskip

\textsc{Elena Guardo, Dipartimento di Matematica e Informatica, Università di Catania, Viale A. Doria 6, 95125,
Catania, Italy}\par\nopagebreak
\textit{E-mail address}: \email{guardo@dmi.unict.it}

\medskip

\textsc{Manuela Muzika Dizdarevi\'c, Faculty of Science, Department of Mathematics, University of Sarajevo, Bosnia}\par\nopagebreak
\textit{E-mail address}: \email{manuela.md@pmf.unsa.ba}

\medskip
\textsc{Maryam Nowroozi, Warwick Mathematics Institute, University of Warwick, CV4 7AL, UK}\par\nopagebreak
\textit{E-mail address}: \email{maryam.nowroozi@warwick.ac.uk}

\medskip

 \textsc{Pierpaola Santarsiero, Dipartimento di Ingegneria Industriale e Scienze Matematiche, Universit\`a Politecnica delle Marche, Via Brecce Bianche
I-60131 Ancona, Italy}\par\nopagebreak
  \textit{E-mail address}: \email{p.santarsiero@staff.univpm.it}

\medskip 

\textsc{Paola Supino, Dipartimento di Matematica e Fisica, Università  Roma Tre, Largo S. L. Murialdo 1, 00146 Roma, Italy}\par\nopagebreak
  \textit{E-mail address}: \email{psupino@uniroma3.it}
}}
\date{}

\begin{document}
\begin{abstract}
 Taylor varieties $\tv^n_{d,e,m}$  arise from  Taylor expansion of rational functions in $n$ variables. Among them, we look for non-defective hypersurfaces.  We prove that the cases  $n=2$ and $m=d+2$   give   new examples of hypersurfaces with identically null Hessian.  
\end{abstract}
\maketitle
 
\section{Introduction}
Let $f \in \mathbb{C}[z_0,\dots,z_N]$ be a homogeneous prime polynomial of degree $d>1$ and $X=V(f) \subset \mathbb{P}^N$ the degree $d$ irreducible hypersurface determined by $f.$ The Hessian matrix of $X$ is defined as \begin{align*}
H_X=\left(\frac{\partial^2 f}{\partial z_i \partial z_j}\right)_{i,j=0,\dots,N} 
\end{align*} and the determinant $h_f=h_X=\det(H_X)$ is called the \emph{Hessian determinant} or \emph{Hessian} of $X.$

The Hessian determinant has classically been used to reveal some geometric properties of $X\subset \mathbb{P}^N$. In fact, it is identically zero if and only if the partial derivatives $f_0;\dots;f_N$ are algebraically dependent, i.e. if and only if there is some nonzero polynomial $g(z_0,\dots,z_N) \in
\mathbb{C}[z_0,\dots,z_N]$ such that $g(\frac{\partial f}{\partial z_0},\dots,\frac{\partial f}{\partial z_N})= 0$. The simplest example is when  $X$ is a cone, for which the derivatives are linearly dependent.  More  generally, the vanishing of the Hessian implies that the tangent spaces of the hypersurface $X$  don't vary as much as expected, as  $X$ is ``flat" in at least one direction. In other words, the smooth points of such an $X$ represent a  generalization of the flex points of plane curves.
To be more precise,  suppose that $X$ is reduced projective variety,  and consider the Gauss map \[
\gamma: X \dashrightarrow (\mathbb{P}^N)^*
\] which associates to any smooth point $p\in X$ its tangent hyperplane in the dual space $\mathbb{P}^{N^*}$. The closure of the image of $\gamma$ is the dual variety $X^*$ of $X$. If $X=V(f)$ is a reduced hypersurface of degree $d>1$, then the Gauss map is given by 
\[
\gamma(p) = \left[ \frac{\partial f}{\partial z_0}(p) , \frac{\partial f}{\partial z_1}(p) , \dots , \frac{\partial f}{\partial z_N}(p) \right]
\]  The Hessian is the Jacobian matrix of the Gaussian map, and describes the differential map \[
d\gamma_p: T_p X \to T_{\gamma(p)} (\mathbb{P}^N)^*.
\]
 When the Hessian vanishes at $p$, that is, $d\gamma_p$ has a non-trivial kernel of dimension $k > 0$,  there are $k$ null principal curvatures:  the tangent hyperplane 
 is infinitesimally stationary along $k$
 independent directions at 
$x$.
When  the Hessian vanishes identically,  the rank of the Gauss map is not maximal globally, the dimension of the dual variety  is less than $n$, and the variety is said with degenerate Gauss image or "tangential degenerate", or "developable". Of course,  $ X$ is smooth if and only if  $\{\frac{\partial f}{\partial z_0}(p) ; \frac{\partial f}{\partial z_1}(p) ; \dots ; \frac{\partial f}{\partial z_N}(p)\} $ form a regular sequence; in particular,
if $X$ is smooth, then its Hessian cannot be not identically zero, hence, having a vanishing Hessian implies  that 
$\Sing(X)\neq \emptyset$. 

Moreover, if $X$ has degenerate Gauss image, then  general fibres of $\gamma$ are linear (\cite{Z}, cap 1, cor. 2.8), and $\Sing(X)$  cuts a general fibre along a codimension $1$ subscheme.
As mentioned above, the Hessian of a hypersurface $V(f)$ is closely related to its local curvature. In  fact,

\begin{theorem*}[Gauss curvature formula, {{\cite[Proposition 7.2.3]{R16}}}]
Let $K(p)$ denote the Gaussian curvature of $X$ at a point $p.$ Then $K(p)=0$ for every $p \in X \setminus \mathrm{Sing}(X) \cup V(z_0)$ if and only if $h_f \equiv 0.$ 
\end{theorem*}

Therefore, which are these hypersurfaces?
If 
 $N\geq 3$, where only cones occur, as  proved by P. Gordon and M. Noether in \cite{GN} in 1876. 
 
 In higher dimensions, the classification of
hypersurfaces not cones with vanishing Hessian play a classical role in the history of algebraic geometry, the first examples being given by P. Gordon and M. Noether in the same paper \cite{GN}. In particular, they gave a complete description of the hypersurfaces in $\mathbb{P}^4$ and a series of examples of hypersurfaces with vanishing Hessian not
cones for any $N\geq 5$. 
An interesting observation is that a null hessian hypersurface can be given as a tangential variety of a subvariety $Y$, where $Y$ is its focal loci, that is the locus where the Hessian determinant drops its rank. Note that in general $Y\subset\Sing (X)$ but not necessarily equality holds (\cite{R16}). 
For its elegance, simplicity and source of further examples the celebrated  Perazzo cubic threefold in $\mathbb{P}^4$, given by  $f (z_0, z_1, z_2, z_3, z_4) = z_0z_3^2+ z_1z_3z_4 + z_2z_4^2=0$ \cite{Pe}
is worth to be mentioned. It is in fact generalized as the 
Perazzo cubic hypersurface in $\mathbb{P}^N$ by the form $f (z_0,\dots,z_N) = z_0z_3^2+ z_1z_3z_4 + z_2z_4^2+z_5^3+\cdots +z_N^3=0$. Besides, the Perazzo hypersurface of degree $d\geq 3$ is a hypersurface of the form $V(f),$ with
$$
f(z_0,\dots,z_N)=z_0A_1(z_k,\dots,z_n)+\cdots +z_{k-1}A_k(z_k,\dots,z_n)+A_{k+1}(z_k,\dots,z_n)
$$
where for $i\in \{1,\dots,k \}$ $A_i $'s are algebraically dependent but linearly independent of degree $d-1$, while $A_{k+1}$ has degree $d$. So in particular $h_f\equiv 0$.

Other known examples are Franchetta's and Permutti's, and their generalizations \cite{CRS, Franchetta1954, G, GR, Permutti1957, Permutti1964}.

A new interesting family of examples of varieties with null Hessian is given by Gondim, Russo, and Staglianò in \cite{GRS}.

Without entering into the picture, we only recall that the study of null Hessian hypersurfaces can also be approached in a more algebraic setting. In fact,  for a given  homogeneous polynomial $f\in \mathbb{C}[z_0,\dots,z_N]$ of degree $d$ one can consider the associated Artinian Gorenstein algebra $ A =
S/ Ann_S(f)$, where $S$ is the
ring of differential operators on $\mathbb{C}[z_0,\dots,z_N]$, and explore the rank of its intern multiplication maps, by means of the so-called strong (or weak) Lefschetz Property. An interesting literature can be found within this setting (see \cite{WB, FMM, G,  MW, RP,  W} just to cite some of them).

Recently, Conca, Naldi, Ottaviani and Sturmfels \cite{TV23} introduced a class of algebraic varieties called \emph{Taylor varieties} $\mathcal{T}^n_{d,e,m}$ whose definition is related to the Taylor expansion of a rational function in $n$ variables. In certain cases, these varieties are determinantal hypersurfaces; moreover,  when $m = d+1$, they have vanishing Hessian (see \cite[Theorem 4.4]{TV23}). Thus, Taylor hypersurfaces provide a new source of varieties with vanishing Hessian, although some instances recover very well known examples (for example, when $n = 2, d = 1, e = 1, m = 2$, one obtains a cone over the Perazzo hypersurface, cf. \cite[Example 4.3]{TV23}).
Motivated by our interest in hypersurfaces with vanishing Hessian, in the following we extend this result by proving that Taylor hypersurfaces of the form $\mathcal{T}^2_{d,e,d+2}$ also have vanishing Hessian.

Our main result is the following.
\begin{theorem}\label{theorem: main}
Fix $d, e \in \mathbb{N}$. If $\mathcal{T}^2_{d,e,d+2} = V(\det(P_T))$ and it is a non-defective hypersurface, 
then the Hessian satisfies $h_f \equiv 0$.
\end{theorem}
The paper is organized as follows. In \Cref{section: taylor} we introduce Taylor varieties, discuss their expected dimension and explain how they are described via the Pad\'e matrix (\Cref{subsection: pade marix}), an essential tool in our work. In \Cref{subsection: example}, we deal with the case $\tv^2_{5,4,7},$ which is the smallest example of a non-defective hypersurface of the form $\tv^n_{d,e,d+2},$ and prove that its Hessian vanishes identically. \Cref{section: main theorem} considers the general case $\tv^2_{d,e,d+2}$ and it is devoted to show our main \Cref{theorem: main}. Finally, \Cref{section: M2 code} contains the  \texttt{Macaulay2} code used to facilitate our computations \cite{M2}.

\subsection*{Acknowledgments} We thank Giorgio Ottaviani for introducing us in Taylor varieties topic and providing the inspiration that sparked this project.  This work has been settled during our visit at the University of Split, Croatia in August 2025, in the frame of Women In Numbers Europe 5. All authors are grateful for the warm hospitality and heartily thank the organizers for generously giving an efficient, friendly environment and for financial support.  Guardo and Supino were partially supported by GNSAGA of INdAM. Guardo was partially supported by Progetto Piaceri 2024-26, Università di Catania and by PRIN 2022, “$0$-dimensional schemes, Tensor Theory and applications” – funded by the European Union Next Generation EU, Mission 4, Component 2 -- CUP: E53D23005670006. Gomes Ribeiro acknowledges support received from the University of Birmingham through a PhD Scholarship whilst writing this work. She also received partial travel funding from the University of Birmingham to attend Women In Numbers Europe 5. Santarsiero was supported by the European Union under NextGenerationEU. PRIN 2022, Prot. 2022E2Z4AK and PRIN 2022 SC-CUP: I53C24002240006.

\section{Taylor varieties}\label{section: taylor}
Let $P, Q \in \mathbb{C}[x_1,\dots,x_n]$ such that $P(0)=Q(0)=1$, $\deg(P)\leq d$ and $\deg(Q)\leq e$. Consider the Taylor expansion at the origin  of the rational function:   
\begin{align}\label{a: rat}
\frac{P}{Q}=1+\sum_{0<|\bm{\gamma}| \leq m}c_{\bm{\gamma}}x^{\bm{\gamma}} +  \{\text{ terms of order } \geq m+1\},   
\end{align}
where $x^\gamma=x_1^{\gamma_1}\cdots x_n^{\gamma_n}$ and $|\bm{\gamma}|=\gamma_1+\cdots+\gamma_n\ > 0$ is the total degree of $x^{\bm{\gamma}}$. The polynomial $\sum_{0<|\bm{\gamma}| \leq m}c_{\bm{\gamma}}x^{\bm{\gamma}}$ is called the \emph{Taylor polynomial} of $P/Q$ of order $m$.   Consider the map 
\begin{align*}
\psi: \C^{{d+n \choose n}-1} \times \C^{{e+n \choose n}-1} &\to \mathbb{C}^{{n+m \choose n}-1}   \\ 
(P,Q) &\mapsto (c_{\gamma})_{0<|\gamma| \leq m}.
\end{align*}

The \textit{Taylor variety} $\mathcal{T}^n_{d,e,m}$ is defined as the projective closure of $\mathrm{Im}(\psi)$ in $\mathbb{P}^N=\mathbb{P}^{{n+m \choose n}-1}$, with variables $[z_0,\dots,z_N]=[c_{\gamma}]_{0\leq |\gamma| \leq m}$ , endowed with monomial basis in decreasing order.
The ideal of this variety is denoted by $\mathcal{I}_{d,e,m}^{n}.$ Being the closure of the image of an irreducible variety via a rational map in the projective space, $\mathcal{T}^n_{d,e,m}$ is irreducible and consequently the ideal $\mathcal{I}_{d,e,m}^{n}$ is prime.

The expected dimension of $\mathcal{T}^n_{d,e,m}$ is  $$
\ed{\tv^n_{d,e,m}}=\min\bigg\{{d+n \choose n}+{e+n \choose n}-2,{m+n \choose n}-1\bigg\}.
$$
If the equality $\ed{\tv^n_{d,e,m}}=\dim(\tv^n_{d,e,m})$ holds, the Taylor variety is said to be \emph{non-defective} and \emph{defective} otherwise.     
Defectiveness actually occur and  in \cite{TV23}, the authors  provide examples. For instance, $\tv^3_{2,2,3}\subset \PP^{19}$ is expected to be a hypersurface but has dimension $17$ (see \cite[Proposition 4.1]{TV23}). Note that The Taylor varieties $\tv^n_{d,e,m}$ and $\tv^n_{e,d,m}$   are birationally isomorphic, by exchanging the role of $T$ with $1/T$ (see \cite[Proposition 5.2]{TV23}). 

Moreover, when the degree $m$ of the Taylor truncation is small with respect to the degrees $d$ and $e$ of the rational function then $\tv^n_{d,e,m}=\PP^{{m+n \choose n}-1}$, so the interesting case is when ${d+n \choose n}+{e+n \choose n}-2<{m+n \choose n}-1$.

\subsection{The Pad\'e matrix}\label{subsection: pade marix} 

Let $x=(x_1,\dots,x_n)$ and choose $T=\sum_{0\leq |\bm{\gamma}| \leq m}c_{\bm{\gamma}}x^{\bm{\gamma}}$  a  polynomial of order $m$ (not identically zero).  Consider the map 

\begin{align}\label{phi}
\varphi_T: \C[x]_{\leq e} &\to \C[x]_{\leq e+m} \to \mathrm{Span}\{M_{d+1,m}\}\subset\C[c_{\gamma}], \\
Q &\mapsto QT \mapsto QT\mid_{M_{d+1,m}}\nonumber
\end{align} where $M_{d+1,m}$ denotes the set of monomials on $x$ of total degree in the interval  $\{d+1,\dots,m\}$ 
 $$
 M_{d+1,m}=\bigg\{\prod_{i=1}^nx_{i}^{a_{i}} : d+1 \leq \sum_{i=1}^na_{i} \leq m\bigg\}.
 $$ 
Notice that $T\in \tv^n_{d,e,m}$ if and only if there exist a non zero polynomial $Q \in \C[x]_{\leq e}$ such that   $QT\mid_{M_{d+1,m}}=0$, so if and only if  kernel of the map $\varphi_T$ is non trivial.

Since $\varphi_T$ is a linear application, we consider the matrix $P_T$ associated to $\varphi_T$  with respect to the standard monomial basis ordered by degree, for convenience chosen increasingly  for the domain of $\varphi_T$ and decreasingly for the image.

We call the $\scriptstyle{\left({m+n \choose n}-{d+n \choose n}\right)\,\times\, {e+n \choose n}}$ matrix $P_T$ associated to the map $\varphi_T$ the \emph{Pad\'e matrix}. It is a matrix whose entries are picked among the coordinates of the point $T\in \mathbb{C}^{n+m \choose n}=\mathbb{C}[x]_{\leq m}$,
in particular, it turns out to have a block  
Hankel structure (see \cite{TV23}).
If $P_T$ is not generically of maximal rank, then $\tv^n_{d,e,m}=\PP^{{n+m \choose n}-1}$, otherwise the ideal $\mathcal{I}^n_{d,e,m}$ of $\tv^n_{d,e,m}$ is related to the ideal generated by the maximal nonvanishing minors of $P_T$. Unfortunately, the latter may be not prime, and $\mathcal{I}^n_{d,e,m}$ may not be Cohen--Macaulay, for instance in the case $\tv^2_{1,1,3}$. It is conjectured in  \cite{TV23} that $\mathcal{I}^n_{d,e,m}$ can be obtained from the ideal of non vanishing maximal minors of $P_T$ saturated by the ideal saturated by the
ideal of maximal non-vanishing minors of the so-called  \emph{reduced Pad\'e matrix} $\hat{P_T}$, which  is defined from $P_T$ by deleting the first column.
 At any rate, the Pad\'e matrix give us a to  to work out examples of Taylor varieties which are hypersurfaces. We remark that there are, for $n=3$,  Taylor varieties  expected to be a hypersurfaces but having codimension two, and there also exist  at least a Taylor variety that is expected to fill its ambient projective space but is in fact a hypersurface (\cite{TV23}, prop 4.1).  

The matrix $P_T$ is a square matrix if and only if \begin{align*}
\ed{\tv^n_{d,e,m}}=\dim \mathbb{P}^{{n+m \choose n}-1}-1,\end{align*}
if $P_T$ has maximal rank and   $\tv^n_{d,e,m}$ is a non-defective hypersurface, then $\tv^n_{d,e,m}=V(\sqrt{\det(P_T)}),$ i.e, $\mathcal{I}^n_{d,e,m}=\langle \sqrt{\det(P_T)} \rangle.$
This case is particularly interesting, for instance: in \cite[Theorem 4.4]{TV23}, the authors proved that the hypersurface $\tv_{d,e,d+1}^n$ has vanishing Hessian. In particular, $\tv_{1,1,2}^2$ is a cubic hypersurface of $\PP^7$, cone over a Perazzo variety, and is the orbit closure of the action
of a 5-dimensional group on $\PP^7$.

In our paper we focus on the case $n = 2$ since Taylor varieties in one variable are classical well-known objects such as projective spaces, secant varieties of rational normal curves, cones or linear sections of secant varieties (of rational normal curves) as described in \cite{TV23}.

Before proving \Cref{theorem: main}, we study in detail the case $n=2$, $d=5, e=4$ and $m=7$. 

\section{The case $n=2,d=5,e=4$}\label{subsection: example}
For this part we fix $n=2,$ $d=5,e=4$ so that $m=7$. We have two variables and to lighten the notation we denote them by $x$ and $y$. Our goal is to show that $h_f \equiv 0$ without explicitly computing it. 
The corresponding Taylor polynomial is of degree 7 and if we denote by $T_i=\sum_{\gamma_1+\gamma_2=i}c_{\gamma_1,\gamma_2}x^{\gamma_1}y^{\gamma_2}$ for $ i \in \{0,\dots,7\}$, then we write the Taylor polynomial $T$ in terms of its homogeneous components as $$
T=\sum_{\gamma_1+\gamma_2 \leq 7}c_{\gamma_1,\gamma_2}x^{\gamma_1}y^{\gamma_2}=T_0+\cdots+T_7.$$   
Following \Cref{subsection: pade marix}, we construct the Pad\'e matrix associated with $\varphi_T$ as in (\ref{phi}). For the domain of $\varphi_T$ we considered the monomial basis in increasing order of degrees and each degree is in decreasing order, while on the target space we considered the monomial basis in decreasing order. We have
\begin{align*}
P_T=\left(\!\begin{array}{ccccccccccccccc}
     c_{7,0}&c_{6,0}&0&c_{5,0}&0&0&c_{4,0}&0&0&0&c_{3,0}&0&0&0&0\\
     c_{6,1}&c_{5,1}&c_{6,0}&c_{4,1}&c_{5,0}&0&c_{3,1}&c_{4,0}&0&0&c_{2,1}&c_{3,0}&0&0&0\\
     c_{5,2}&c_{4,2}&c_{5,1}&c_{3,2}&c_{4,1}&c_{5,0}&c_{2,2}&c_{3,1}&c_{4,0}&0&c_{1,2}&c_{2,1}&c_{3,0}&0&0\\
     c_{4,3}&c_{3,3}&c_{4,2}&c_{2,3}&c_{3,2}&c_{4,1}&c_{1,3}&c_{2,2}&c_{3,1}&c_{4,0}&c_{0,3}&c_{1,2}&c_{2,1}&c_{3,0}&0\\
     c_{3,4}&c_{2,4}&c_{3,3}&c_{1,4}&c_{2,3}&c_{3,2}&c_{0,4}&c_{1,3}&c_{2,2}&c_{3,1}&0&c_{0,3}&c_{1,2}&c_{2,1}&c_{3,0}\\
     c_{2,5}&c_{1,5}&c_{2,3}&c_{0,5}&c_{1,4}&c_{2,3}&0&c_{0,4}&c_{1,3}&c_{2,2}&0&0&c_{0,3}&c_{1,2}&c_{2,1}\\
     c_{1,6}&c_{0,6}&c_{1,5}&0&c_{0,5}&c_{1,4}&0&0&c_{0,4}&c_{1,3}&0&0&0&c_{0,3}&c_{1,2}\\
     c_{0,7}&0&c_{0,6}&0&0&c_{0,5}&0&0&0&c_{0,4}&0&0&0&0&c_{0,3}\\
     c_{6,0}&c_{5,0}&0&c_{4,0}&0&0&c_{3,0}&0&0&0&c_{2,0}&0&0&0&0\\
     c_{5,1}&c_{4,1}&c_{5,0}&c_{3,1}&c_{4,0}&0&c_{2,1}&c_{3,0}&0&0&c_{1,1}&c_{2,0}&0&0&0\\
     c_{4,2}&c_{3,2}&c_{4,1}&c_{2,2}&c_{3,1}&c_{4,0}&c_{1,2}&c_{2,1}&c_{3,0}&0&c_{0,2}&c_{1,1}&c_{2,0}&0&0\\
     c_{3,3}&c_{2,3}&c_{3,2}&c_{1,3}&c_{2,2}&c_{3,1}&c_{0,3}&c_{1,2}&c_{2,1}&c_{3,0}&0&c_{0,2}&c_{1,1}&c_{2,0}&0\\
     c_{2,4}&c_{1,4}&c_{2,3}&c_{0,4}&c_{1,3}&c_{2,2}&0&c_{0,3}&c_{1,2}&c_{2,1}&0&0&c_{0,2}&c_{1,1}&c_{2,0}\\
     c_{1,5}&c_{0,5}&c_{1,4}&0&c_{0,4}&c_{1,3}&0&0&c_{0,3}&c_{1,2}&0&0&0&c_{0,2}&c_{1,1}\\
     c_{0,6}&0&c_{0,5}&0&0&c_{0,4}&0&0&0&c_{0,3}&0&0&0&0&c_{0,2}
     \end{array}\!\right).  
\end{align*}

\begin{remark}\label{remark: the pade of the example is nonsingular}
To ensure that the generic rank of $P_T$ is $15$ and as a consequence that $\tv^2_{5,4,7}$ is non-defective, it is enough to compute its determinant at a random choice of coefficients $c_\gamma$'s and verify that it is not zero. One can easily replicate the computation using the Macaulay2 code of \Cref{section: M2 code}, \cite{M2}.

Hence, the Taylor variety $\tv^2_{5,4,7}\subset \PP^{35}$ is the hypersurface defined by the vanishing of the degree fifteen polynomial $f:=\det(P_T)$. Note that the ideal generated by $f$ is prime, as pointed out in  \Cref{section: taylor}.
\end{remark}

Let $C_j: \mathbb{C}[x,y]_{7-j} \to \mathbb{C}[x,y]_{6,7}$ denote the matrix of the linear map multiplication by $T_j.$ Notice that $C_j$ has entries the coefficients $c_{\gamma}$ of the polynomial $T$ and size $15\times (7-j+1).$  

We rewrite the Pad\'e matrix as $$P_T=\begin{bmatrix}
\underbrace{C_7}_{15\times 1} & \underbrace{C_6}_{15\times 2} & \underbrace{C_5}_{15\times 3} & \underbrace{C_4}_{15\times 4} & \underbrace{C_3}_{15\times 5} 
\end{bmatrix}_{15 \times 15},$$
and we remark that for $j=3,4,5,6,7$ the blocks are given by $C_j=C_j(c_{\mu,\nu} \mid \mu+\nu=j,j-1)$.

For $j \in \{4,5,6,7\},$ let $d_j=j-3.$ We define vectors of new independent variables $\lambda^j=(\lambda_{\alpha}^j)_{|\alpha|=d_j},$ where the entries are indexed by the monomials in $x,y$ of degree $d_j$:
\begin{align*}
 \lambda^4=\begin{pmatrix}
        \lambda^4_{1,0}\\ \lambda^4_{0,1}
    \end{pmatrix}, \,     \lambda^5=\begin{pmatrix}
        \lambda^5_{2,0}\\ \lambda^5_{1,1}\\ \lambda^5_{0,2}
    \end{pmatrix}, \,  \lambda^6=\begin{pmatrix}
        \lambda^6_{3,0}\\ \lambda^6_{2,1}\\ \lambda^6_{1,2}\\ \lambda^6_{0,3}
    \end{pmatrix}, \,  \lambda^7=\begin{pmatrix}
        \lambda^7_{4,0}\\ \vdots \\ \lambda^7_{0,4}
    \end{pmatrix}. 
\end{align*}
We replace the first column $C_j^1$ of $C_j$ with $C_j^1+C_{d+2-e}\begin{bmatrix}
\lambda^j & \mathbf{0}    
\end{bmatrix}^T,$ where $\mathbf{0}$ is the zero vector of length $4-d_j.$ For example, for $j=6,$ this corresponds to adding to $C_6^1$ a linear combination of the columns of $C_3$ 
\begin{align*}
C_6^1 \rightarrow C_6^1+\lambda_{30}^6 C_3^1+\lambda_{21}^6 C_3^2 + \lambda_{12}^6 C_3^3 + \lambda_{03}^6 C_3^4.    
\end{align*}
We also modify the remaining columns via $C_j^i \rightarrow C_j^i+C_{d+2-e}\begin{bmatrix}
\bm{0}_{i-1} & \lambda^j & \bm{0}_{4-d_j-i+1}  
\end{bmatrix}^T,$ where $\bm{0}_{i-1}$ is a row vector of length $i-1$ and $\bm{0}_{4-d_j-i+1}$ is a row vector of length $4-d_j-i+1.$ For instance, the second column of $C_6$ is transformed into 
\begin{align*}
C_6^2 \rightarrow C_6^2+\lambda_{30}^6C_3^2+\lambda_{21}^6 C_3^3 + \lambda_{12}^6 C_3^4 + \lambda_{03}^6 C_3^5,  
\end{align*}
and similarly for the others. We can also express the modification in terms of matrix multiplication: \begin{align*}
   & C_4 \rightarrow C_4+C_3\begin{pmatrix}
        \lambda^4_{1,0} & 0 & 0 & 0\\
        \lambda^4_{0,1} & \lambda^4_{1,0} & 0 & 0\\
        0 & \lambda^4_{0,1} & \lambda^4_{1,0} & 0\\
        0 & 0 & \lambda^4_{0,1} & \lambda^4_{1,0}\\
         0 & 0 & 0 & \lambda^4_{0,1}\\
    \end{pmatrix}=:C_4+C_3(\lambda^4), \\ &C_5 \rightarrow C_5+C_3\begin{pmatrix}
        \lambda^5_{2,0} & 0 & 0 \\
        \lambda^5_{1,1} & \lambda^5_{2,0} & 0 \\
        \lambda^5_{0,2} & \lambda^5_{1,1} & \lambda^5_{2,0} \\
        0 & \lambda^5_{0,2}& \lambda^5_{1,1} \\
         0 & 0 & \lambda^5_{0,2} \\
    \end{pmatrix}=:C_5+C_3(\lambda^5), \\
 & C_6 \rightarrow C_6+C_3\begin{pmatrix}
        \lambda^6_{3,0} & 0  \\
        \lambda^6_{2,1} & \lambda^6_{3,0}  \\
        \lambda^6_{1,2} & \lambda^6_{2,1}  \\
        \lambda^6_{0,3} & \lambda^6_{1,2} \\
         0 & \lambda^6_{0,3}  \\
    \end{pmatrix}=:C_6+C_3(\lambda^6), \\   &C_7 \rightarrow C_7+C_3\begin{pmatrix}
        \lambda^7_{4,0}  \\
        \vdots  \\
        \lambda^7_{0,4}  \\
    \end{pmatrix}=:C_7+C_3(\lambda^7).
\end{align*}
The new matrix $P'_T$ obtained from $P_T$ with the transformations above is
$$
P'_T=\begin{bmatrix}
C_7+C_3(\lambda^7)&C_6+C_3(\lambda^6)& C_5+C_3(\lambda^5) & C_4+C_3(\lambda^4) & C_3    
\end{bmatrix}_{15 \times 15},
$$
and we denote its determinant by $f':=\det P'_T$. Since $P_T'$ is obtained from $P_T$ via elementary operations on the columns, we have that $f=\det P_T=\det P'_T=f'$, where we notice that $f=f(c_{\mu,\nu})$ while $f'=f'(c_{\mu,\nu},\lambda^k)$. 

Looking at the partial derivatives of $f$ and $f'$ with respect to $\lambda^k$'s, since $f$ does not depend on those variables, we have
$$
0=\frac{\partial f}{\partial \lambda^j_{\alpha}}=\frac{\partial f'}{\partial \lambda^j_{\alpha}}=\sum_{\beta}\frac{\partial f}{\partial c_{\alpha+\beta}}c_{\beta}.
$$ 
For instance, if $j=4,$ we have two equations \begin{align*}
0&=\frac{\partial f}{\partial \lambda_{1,0}^4}=\frac{\partial f'}{\partial \lambda_{1,0}^4}=c_{3,0}\frac{\partial f}{\partial c_{4,0}}+c_{2,1}\frac{\partial f}{\partial c_{3,1}}+c_{1,2}\frac{\partial f}{\partial c_{2,2}}+c_{0,3}\frac{\partial f}{\partial c_{1,3}}+c_{2,0}\frac{\partial f}{\partial c_{3,0}}+c_{1,1}\frac{\partial f}{\partial c_{2,1}}+c_{0,2}\frac{\partial f}{\partial c_{1,2}} \\ 0&=\frac{\partial f}{\partial \lambda_{0,1}^4}=\frac{\partial f'}{\partial \lambda_{0,1}^4}=c_{3,0}\frac{\partial f}{\partial c_{3,1}}+c_{2,1}\frac{\partial f}{\partial c_{2,2}}+c_{1,2}\frac{\partial f}{\partial c_{1,3}}+c_{0,3}\frac{\partial f}{\partial c_{0,4}}+c_{2,0}\frac{\partial f}{\partial c_{2,1}}+c_{1,1}\frac{\partial f}{\partial c_{1,2}}+c_{0,2}\frac{\partial f}{\partial c_{0,3}}   
\end{align*} from which we obtain a matrix \begin{align}
M_4=\begin{bmatrix} \label{eq: eqs}
f_{4,0} & f_{3,1} & f_{2,2} & f_{1,3} & f_{3,0} & f_{2,1} & f_{1,2} \\
f_{3,1} & f_{2,2} & f_{1,3} & f_{0,4} & f_{2,1} & f_{12} & f_{0,3}
\end{bmatrix}.    
\end{align} By doing the same process for $j=5,6,7,$ we obtain a bigger matrix \begin{align*}
\mathcal{M}=\begin{bmatrix}
M_7 \\
M_6 \\
M_5 \\
M_4
\end{bmatrix}_{14 \times 7}   
\end{align*} that operates in the vector $\bm{c}=(c_{3,0},c_{2,1},c_{1,2},c_{0,3},c_{2,0},c_{1,1},c_{0,2}).$ 
The equations coming from the partial derivatives such as in \cref{eq: eqs} read as the linear system $\mathcal M \cdot \bm{c}=0$ which is a non-trivial polynomial identity. Since $\mathcal{T}^2_{5,4,7} = V(f)$ is a hypersurface,  from Remark \ref{remark: the pade of the example is nonsingular},  $f = \det(P_T)$ is an irreducible polynomial and
$f$ is not identically zero and its gradient 
$(f_\gamma)$ is nonzero at a generic point 
of $\mathcal{T}^2_{5,4,7}$, so in particular $\mathcal M$ is not the zero matrix and its rank is at least $1$ at a general point. On the other hand, we just observed that $\mathrm{rk} (\mathcal M)<7$, since it has non-zero kernel. Hence,  
the maximal minors of $\mathcal{M}$ vanish, which means that the image of the polar map \begin{align*}
\omega \mapsto (f_{\gamma})_{c_{\gamma} \in \omega} ,   
\end{align*} (where $\omega$ is the vector of coefficients $c_{\gamma}$ of $T$) is contained in $V(I_{7 \times 7}(\mathcal{M})),$ where $I_{7 \times 7}(\mathcal{M})$ denotes the ideal of maximal minors of $\mathcal{M}.$ In particular, the first-order partial derivatives of $f$ satisfy a polynomial relation. This is equivalent to having $h_f\equiv 0$ (see \cite[\S 7.2.1]{R16}).

\begin{remark}
Opposed to \cite[Theorem 4.4]{TV23}, the modifications made in $P_T$ as described above do not give a a well defined action, since a coefficient $c_{\gamma}$ may appear in more than one block of $P_T$ and be replaced by different combinations in each block. It happens because the blocks $C_j$ in our example (and more generally for $m-d=2$) contain coefficients $c_{\gamma}$ of two different degrees whereas in \cite[Theorem 4.4]{TV23} the blocks have coefficients corresponding to a unique fixed degree. Nevertheless, the modifications are still well defined as elementary operations in the columns of the matrices and are still effective for our purposes (see the forthcoming \Cref{lem:invariance}).
\end{remark}

\section{Main result}\label{section: main theorem} 
The purpose of this section is to prove our main result. We start by proving a necessary condition for a Taylor variety to be a hypersurface.
\begin{proposition}\label{nonsingular}
Let $\mathcal{T}^{n}_{d,e,m} \subset \mathbb{P}^N$ be a Taylor variety. If $\mathcal{T}^{n}_{d,e,m}$ is a non-defective hypersurface, then the associated Padé matrix $P_T$ is non-singular, i.e., $\det(P_T) \not\equiv 0$.
\end{proposition}

\begin{proof}
By definition, a non-defective hypersurface in $\mathbb{P}^N$ has dimension $N-1$ and it is the zero set of a single irreducible polynomial $f \in \mathbb{C}[c_0, \dots, c_N]$. In the study of Taylor varieties, the conditions for a point to lie on $\mathcal{T}^{n}_{d,e,m}$ are given by the vanishing of the determinant of the Padé matrix $P_T$, which encodes the syzygies relations of the Taylor polynomial. 

If $P_T$ were singular for a general choice of coefficients, then $\det(P_T)$ would be the zero polynomial. This would imply that the variety has codimension greater than one, contradicting the hypothesis that $\mathcal{T}^{n}_{d,e,m}$ is a hypersurface. Therefore, the non-singularity of $P_T$ is a necessary condition for the variety to be a hypersurface whose defining equation is $f = \det(P_T)$.
\end{proof}

The following lemma provides a technical construction that will be used to prove our main theorem.
\begin{lemma} \label{lem:invariance}
Let $P_T$ be the Padé matrix associated with the Taylor variety $\mathcal{T}^{n}_{d,e,m}$ and let $f = \det(P_T)$ be the polynomial defining the hypersurface. Let $C_j$ denote the column blocks of $P_T$. If a new matrix $P_T'$ is constructed by replacing each column $C_j^i$ (for $j > d-e+2$) of the block $C_j$ with the linear combination
\begin{equation}
    C_j^i \mapsto C_j^i + \sum_{|\alpha|=j} \lambda_{\alpha}^j C_{d-e+2}^\alpha
\end{equation}
    where $\lambda_{\alpha}^j$ are independent auxiliary variables, 
    then:
\begin{equation}
    \det(P_T') = \det(P_T) = f.
\end{equation}
Consequently, the determinant $f$ is invariant with respect to the variables $\lambda_{\alpha}^j$.
\end{lemma}

\begin{proof}
The determinant of a matrix is an alternating multilinear function of its columns. A fundamental property of the determinant is that it remains unchanged if a linear combination of other columns is added to a specific column (elementary column operations).

In the construction of $P_T'$, we add to the columns of the blocks $C_j$ (where $j > d-e+2$) a linear combination of the columns belonging to the block $C_{d-e+2}$. Since the columns of $C_{d-e+2}$ are already present in the original matrix $P_T$ in distinct positions, this operation preserves the value of the determinant for any value of the coefficients $\lambda_{\alpha}^j$. 
\end{proof}

\begin{remark}
We observe that the resulting polynomial $\det(P_T')$ obtained as in \Cref{lem:invariance} cannot explicitly contain the terms $\lambda_{\alpha}^j$. Formally, the invariance $\det(P'_T) = \det(P_T) = f$ implies $\frac{\partial f}{\partial \lambda^j_\alpha} = 0$,
which is a polynomial identity of linear relations of the form $\mathcal M \cdot \mathbf{c} = 0$, as we shall see in detail  in the proof of our main \Cref{theorem: main}. This will give us an algebraic relation among the partial derivatives of $f$ and help us to conclude that $h_f\equiv 0$.
 \end{remark}

We are now ready to prove our main \Cref{theorem: main}. For convenience we restate it below. We also refer to the running example of \Cref{subsection: example} which may be of help in illustrating the steps of the proof. 

\begingroup
\def\thetheorem{\ref{theorem: main}}
\begin{theorem}
Fix $d, e \in \mathbb{N}$. If $\mathcal{T}^2_{d,e,d+2} = V(\det(P_T))$  
 and it is a non-defective hypersurface, 
then the Hessian satisfies $h_f \equiv 0$.
\end{theorem}
\addtocounter{theorem}{-1}
\endgroup

\begin{proof}
We work with $n=2$ variables $x,y$ and set $m = d+2$. 
Write the Taylor polynomial in terms of its homogeneous components as
$$
T = \sum_{|\gamma| \leq d+2} c_\gamma x^\gamma = T_0 + T_1 + \cdots + T_{d+2},
$$
where $T_i = \sum_{|\gamma|=i} c_\gamma x^\gamma$.
By assumption, the Pad\'e matrix  has the following form where for the domain of $\varphi_T$ we considered the monomial basis in increasing order of degrees and each degree is in decreasing order, while on the target space we considered the monomial basis in decreasing order:

\begin{align*}
P_T = \left( \begin{array}{ccccccccccc} 
c_{d+2,0} & c_{d+1,0} & 0 & c_{d,0} & 0 & 0 & \dots & \dots & c_{d-e+2,0} & \dots & 0 \\
c_{d+1,1} & c_{d,1} & c_{d+1,0} & c_{d-1,1} & c_{d,0} & 0 & \dots & \dots & c_{d-e+1,1} & \dots & \vdots\\
c_{d,2} & c_{d-1,2} & c_{d,1} & \vdots & c_{d-1,1} & c_{d,0} & \ddots & \ddots & \vdots & \ddots & \vdots \\
\vdots & \vdots & \vdots & \vdots & \vdots & c_{d-1,1} & \ddots & \ddots & c_{0,d-e+2} & \ddots & 0 \\
\vdots & \vdots & \vdots & \vdots & \vdots & \vdots & \dots & \dots & 0& \ddots & c_{d-e+2,0} \\
\vdots & \vdots & \vdots & c_{0,d} & \vdots & \vdots & \dots & \dots & 0& \ddots & \vdots \\
c_{1,d+1}  & c_{0,d+1} & c_{1,d} & 0 & c_{0,d} & \vdots & \dots & \dots & \vdots & \dots & \vdots \\
c_{0,d+2} & 0 & c_{0,d+1} & 0 & 0 & c_{0,d} & \dots & \dots & 0 & \dots & c_{0,d-e+2} \\
c_{d+1,0} & c_{d,0} & 0 & c_{d-1,0} & 0 & 0 & \dots & \dots & c_{d-e+1,0} & \dots & 0 \\
c_{d,1} & c_{d-1,1} & c_{d,0} & c_{d-2,1} & c_{d-1,0} & 0 & \dots & \dots & c_{d-e,1} & \ddots & \vdots \\
\vdots & \vdots & \vdots & \vdots & \vdots & \vdots & \ddots & \ddots & \vdots & \ddots & 0 \\
\vdots & \vdots & \vdots & \vdots & \vdots & \vdots & \ddots & \ddots & c_{0,d-e+1} & \ddots & c_{d-e+1,0} \\
\vdots & \vdots & \vdots & c_{0,d-1} & \vdots & \vdots & \dots & \dots & 0 & \dots & \vdots \\
c_{1,d} & c_{0,d} & \vdots & 0 & c_{0,d-1} & \vdots & \dots & \dots & \vdots & \dots & \vdots \\
c_{0,d+1} & 0 & c_{0,d} & 0 & 0 & c_{0,d-1} & \dots & \dots & 0 & \dots & c_{0,d-e+1}\\
\multicolumn{1}{c}{\underbrace{\hphantom{c_{d+2,0}}}_{1}} &
\multicolumn{2}{c}{\underbrace{\hphantom{c_{d+1,0}\quad 0}}_{2}} &
\multicolumn{3}{c}{\underbrace{\hphantom{ \quad 0 \quad c_{d,0}\quad 0\quad 0}}_{3}} & 
\dots & \dots&
\multicolumn{3}{c}{\underbrace{\hphantom{\quad 0 \ c_{d-e+2,0}\quad \dots \quad 0\quad 0 }}_{e+1}}
\end{array} \right)
\end{align*}
or, equivalently, \[P_T = \left[C_{d+2}\quad  C_{d+1} \quad C_d \quad \cdots \quad C_{d-e+2}\right]\] 
and it is square of size $(2d+5) \times (2d+5)$, where $C_j$ denotes 
the matrix representing multiplication by $T_j$.
We set $f := \det(P_T)$. 
The rows of $P_T$ are indexed by monomials $x^\rho$ 
with $|\rho| \in \{d+1, d+2\}$, 
and the columns are indexed by monomials $x^\sigma$ 
with $|\sigma| \in \{0, 1, \ldots, e\}$. 
The $(\rho, \sigma)$ entry of the block $C_j$ equals 
$c_{\rho-\sigma}$ if $\sigma \leq \rho$ and $|\rho - \sigma| = j$ 
or $|\rho-\sigma| = j-1$, and zero otherwise.

We now introduce the following Column operations.
For each $j = d-e+3, \ldots, d+2$, set $d_j := j - (d-e+2)$, 
so that $d_j$ ranges from $1$ to $e$. 
Introduce new variables
$$
\lambda^j = (\lambda_{\alpha}^j)_{\alpha \in M_{d_j,d_j}}=(\lambda^j_\alpha)_{|\alpha| = d_j},
$$
indexed by monomials of degree $d_j$ in $x, y$, 
so $\lambda^j$ has $d_j + 1$ components. Notice that their sizes add to $e-d_j$ which add to $e+1$ with the size of $\lambda^j.$

For each such $j$ and for each $i = 1, \ldots, d-j+3$, 
replace the $i$-th column $C^i_j$ of the block $C_j$ by
$$
C^i_j \;\longmapsto\; C^i_j + C_{d-e+2}
\begin{pmatrix} \bm{0}_{i-1} \\ \lambda^j \\ \bm{0}_{e-d_j-i+1} \end{pmatrix},
$$
where $0_k$ denotes a zero vector of length $k$, 
so that the vector 
$(\bm{0}_{i-1},\, \lambda^j,\, \bm{0}_{e-d_j-i+1})^T$ 
has $e+1$ entries matching the number of columns of $C_{d-e+2}$.

Denote by $P'_T$ the matrix obtained from $P_T$ after performing all these column operations.
Since these are elementary column operations, by \Cref{lem:invariance}, we have
$$
f = \det(P_T) = \det(P'_T).
$$

Since $f = \det(P'_T)$ and $f$ does not depend 
on the auxiliary variables $\lambda^j_\alpha$, we have
$$
0 = \frac{\partial f}{\partial \lambda^j_\alpha} 
= \frac{\partial \det(P'_T)}{\partial \lambda^j_\alpha}
$$
for all $j = d-e+3, \ldots, d+2$ and all $\alpha$ with $|\alpha| = d_j$.

By multilinearity of the determinant in the columns of $P'_T$, 
the derivative $\frac{\partial \det(P'_T)}{\partial \lambda^j_\alpha}$ 
is computed by replacing the $i$-th column of block $C_j$ 
with the vector whose $\rho$-th entry is 
$(C_{d-e+2})_{\rho,\alpha} = c_{\rho-\alpha}$ 
(when $\alpha \leq \rho$ and $|\rho-\alpha| \in \{d-e+1,\, d-e+2\}$, 
zero otherwise), and taking the determinant. 
Expanding along this column and applying the chain rule 
to pass from cofactors to partial derivatives of $f$, 
evaluated at $\lambda^j = 0$, we obtain
\begin{equation}\label{eq:relations}
0 = \sum_{\substack{\beta \,:\, |\beta| \in \{d-e+1,\, d-e+2\} \\ 
|\alpha+\beta| \in \{j-1,j\}}} 
c_\beta \cdot f_{\alpha+\beta},
\end{equation} 
\noindent for every $j = d-e+3,\ldots, d+2$ and every $\alpha$ 
with $|\alpha| = d_j$, 
where $f_\gamma := \frac{\partial f}{\partial c_\gamma}$.

The two ranges $|\beta| = d-e+1$ (giving $|\alpha+\beta|=j-1$) 
and $|\beta| = d-e+2$ (giving $|\alpha+\beta| = j$) 
correspond precisely to the two groups of row indices of $P_T$, contributing to the block $C_j$. 
This is the reason both values of $|\beta|$ appear.

The next step of the proof focus on the introduction of the Hankel matrix $M$ and look at its kernel as in the proof of Theorem 4.4 in \cite{TV23}.
For each $j = d-e+3, \ldots, d+2$, 
define the matrix $M_j$ of size $(d_j+1) \times {(2d-2e+5)} $ by 
$$(M_j)_{\alpha,\beta} = f_{\alpha+\beta}$$

\noindent where rows are indexed by $\alpha$ with $|\alpha| = d_j$ 
and columns are indexed by $\beta$ 
with $|\beta| \in \{d-e+1, d-e+2\}$. 

Construct the following matrix
$$
\mathcal{M} := \begin{bmatrix}
M_{d+2} & \cdots & M_{d+3-e}   
\end{bmatrix}^T$$
which has $\binom{e+2}{2}-1$ rows 
and $2d-2e+5$ columns.  
The relation \eqref{eq:relations} reads exactly

$$
\mathcal{M} \cdot \mathbf{c} = 0,
$$
where $\mathbf{c} = (c_\gamma)_{|\gamma| \in \{d-e+2,d-e+1\}}$ as a polynomial identity. 

Since by our hypotheses $\mathcal{T}^2_{d,e,d+2} = V(f)$ is a hypersurface 
and $f = \det(P_T)$ is an irreducible polynomial, 
$f$ is not identically zero and its gradient 
$(f_\gamma)$ is nonzero at a generic point 
of $\mathcal{T}^2_{d,e,d+2}$. Notice that  this implies that $\mathcal M$ is not the zero matrix and its rank is at least $1$ at a general point. Moreover, the relation (\ref{eq:relations}) gives nontrivial kernel.

It remains to prove the vanishing of the Hessian. 

 Since the Pad\'e matrix is square, $2d+5={e+2 \choose 2}$ and therefore $2d-2e+5={e+2 \choose 2}-2e<{e+2 \choose 2}$. We have $\mathrm{rk}(\mathcal M)<\min\{\binom{e+2}{2}-1,2d-2e+5\} =2d-2e+5$ 
and at a general point,
\[
\dim \operatorname{Span}\{f_\gamma\} \leq\mathrm{rk}(\mathcal M) < 2d-2e+5.
\]
Notice that the polynomial $f = \det(P_T)$ depends on all 
coefficients $c_\gamma$ with $|\gamma| \leq d+2$, 
so the polar map takes values in 
$\mathbb{P}^{\binom{d+4}{2}-1}$. Therefore, if we denote by $\omega$ the vector of coefficients of $T$, the image of the polar map
$$
w \longmapsto (f_\gamma : c_\gamma \in w) 
$$
is contained in a proper
linear subspace $V(I_{(2d-2e+5)\times(2d-2e+5)}(\mathcal M))$ of $\PP^{|\omega| -1}=\mathbb{P}^{\binom{d+4}{2}-1}$ given by the vanishing of the maximal minors of $\mathcal M$ and in particular it is not dominant. Hence, we found a polynomial relation among the first-order partial derivatives of $f$ and this implies $ h_f\equiv 0$ (see \cite[\S 7.2.1]{R16}).
Equivalently, since the Hessian matrix $H_f$ is the Jacobian matrix 
of the polar map, its determinant satisfies
$$
h_f = \det(H_f) \equiv 0,
$$
which completes the proof.
\end{proof}

\section{Macaulay2 computations}\label{section: M2 code} 
In this part we include the \texttt{Macaulay2} \cite{M2} code that we used for our computations. The code runs on \texttt{Macaulay2 1.25}. The main function composing our code is \texttt{PadeMatrix}. It takes as input non-negative integers $d,e,m$ and it outputs the Pad\'e matrix $P_T$ associated with the map $\varphi_T$ for the particular case $n=2$. The Pad\'e matrix is constructed by considering monomial basis in increasing order for the domain and decreasing order on the image.
\begin{lstlisting}[language = code] 
PadeMatrix = (d,e,m) -> (
    L := toList flatten(
    apply(0..m, i -> toList apply(0..(m-i), j -> (i,j))));
    C := QQ[apply(L,x->c_x)]; 
    S := C[x,y];
    cv := toList apply(0..m,i->reverse(
        select(flatten entries vars C,
            x->sum toList((baseName x)#1)==i))
        );
    T := sum toList apply(0..m,i->(sub(matrix{cv_i},S)*transpose(basis(i,S)))_(0,0)); 
    mult := toList apply(0..e,i->apply(reverse flatten entries basis(i,S),x->x*T));  
    -- remove "reverse" to take monomials in each degree of the domain in decreasing order
    imMap := flatten mult; 
    s := toList((d+1)..m); 
    cols := matrix apply(s,y->apply(imMap,x->part({y,1},x))); 
    PadeMatrix := (coefficients cols)_1 
)
\end{lstlisting}
For instance, the Pad\'e matrices relative to \cite[Examples 3.3, 4.3, 4.5]{TV23} are obtained as follows after loading the 
\texttt{PadeMatrix} function.
\begin{lstlisting}[language = code]
PadeMatrix(2,1,3) -- [Example 3.3, CNOS23]  
PadeMatrix(1,1,2) -- [Example 4.3, CNOS23]
PadeMatrix(4,2,5) -- [Example 4.5, CNOS23]  
\end{lstlisting}
The Pad\'e matrix of \Cref{subsection: example} is obtained with the command \texttt{M = PadeMatrix(5,4,7)} where the command \texttt{reverse} in \texttt{mult} is removed. To ensure that $M$ is non-singular, we substitute random values to the entires of $M$ and verify that the corresponding determinant is non-zero (see \Cref{remark: the pade of the example is nonsingular}):
\begin{lstlisting}[language = code]
L = flatten entries vars coefficientRing ring M;
det sub(M,apply(L,i->i=>random(QQ)))    
\end{lstlisting}

\bibliographystyle{alpha}
\bibliography{references.bib}

@article{Franchetta1954,
  author  = {Franchetta, A.},
  title   = {Sulle forme algebriche di $S_4$ aventi l'hessiana indeterminata},
  journal = {Rendiconti di Matematica e delle sue Applicazioni},
  volume  = {13},
  year    = {1954},
  pages   = {1--6}
}

@article{CRS,
  title={Homaloidal hypersurfaces and hypersurfaces with vanishing Hessian},
  author={Ciliberto, C. and Russo, F. and Simis, A.},
  journal={Advances in Mathematics},
  volume={218},
  pages={1759--1805},
  year={2008}
}

@article{TV23,
  TITLE = {{Taylor Polynomials of Rational Functions}},
  AUTHOR = {Conca, A. and Naldi, S. and Ottaviani, G. and Sturmfels, B.},
  URL = {https://hal.science/hal-04526513},
  JOURNAL = {{Acta Mathematica Vietnamica}},
  PUBLISHER = {{Springer Singapore}},
  YEAR = {2023},
  MONTH = Nov,
  DOI = {10.1007/s40306-023-00514-4},
  HAL_ID = {hal-04526513},
  HAL_VERSION = {v1},
}

@article{GRS,
    author = {Gondim, R. and Russo, F. and Staglianò, G.},
    title = {Hypersurfaces with vanishing hessian via dual Cayley trick},
   journal={Journal of Pure and Applied Algebra},
  volume={224,3},
  pages={ 1215-1240},
  year={2020}
}

@article{GN,
  title={Ueber die algebraischen Formen, deren Hessesche Determinante identisch verschwindet},
  author={Gordon, P. and Noether, M.},
  journal={Mathematische Annalen},
  volume={10},
  pages={547--568},
  year={1876}
}

@article{Pe,
  title={ Sopra una forma cubica con 9 rette doppie dello spazio a cinque dimensioni, e i
correspondenti complessi cubici di rette nello spazio ordinario},
 author={Perazzo, U.},
  journal={Atti della Accademia delle Scienze di Torino},
  volume={136,4},
  pages={ 891–895},
  year={1901}
}

@article{Permutti1957,
  author  = {Permutti, R.},
  title   = {Su certe forme a hessiana indeterminata},
  journal = {Ricerche di Matematica},
  volume  = {6},
  year    = {1957},
  pages   = {3--10}
}

@article{Permutti1964,
  author  = {Permutti, R.},
  title   = {Su certe classi di forme a hessiana indeterminata},
  journal = {Ricerche di Matematica},
  volume  = {13},
  year    = {1964},
  pages   = {97--105}
}

@Misc{M2,
          author = { Grayson, D. R. and Stillman, M. E. },
          title = {Macaulay2, a software system for research in algebraic geometry},
          howpublished = {Available at \url{https://macaulay2.com/}}
        }

@book{R16,
    author = {Russo, F.},
    title = {On the Geometry of Some Special Projective Varieties},
    publisher = {Springer Cham},
    year = {2016},
    doi = {https://doi.org/10.1007/978-3-319-26765-4},
    series = {Lecture Notes of the Unione Matematica Italiana}
}

@book{Z,
    author = {Zac, F.L.},
    title = {Tangents and Secants of Algebraic Varieties},
    publisher = { American Mathematical Society},
    year = {1993 },
     series = { Translations of Mathematical Monographs, vol. 127}
}

@article{RP,
  title={Perazzo hypersurfaces and the weak Lefschetz property},
  author={Mirò-Roig, R. M. and Pérez, J.},
  journal={Journal of Algebra},
  volume={646},
  pages={357--375},
  year={2024}
}

@article{FMM,
    title={Perazzo 3-folds and the weak Lefschetz property},
    author={Fiorindo, L. and Mezzetti, E. and Miró-Roig, R.M.},
  journal={Journal of Algebra},
  volume={626},
  pages={56--81},
  year={2023}
}

@article{G,
  title={On higher Hessians and the Lefschetz properties},
  author={Gondim, R.},
  journal={Journal of Algebra},
  volume={489},
  pages={241--263},
  year={2017}
}

@article{GR,
        title = {On cubic hypersurfaces with vanishing hessian},
        author={Gondim, R. and Russo, F.},
journal={Journal of Pure and Applied Algebra},
 volume={219,4},
  pages={779--806},
  year={2015}
}

@article{MW,
  title={Lefschetz elements of artinian Gorenstein algebras and Hessians of homogeneous polynomials},
  author={Maeno, T. and Watanabe, J.},
  journal={Illinois J. Math.},
  volume={53},
  pages={-593--603},
  year={2009}
}

@article{W,
  title={ On the Theory of Gordan-Noether on Homogeneous Forms with Zero Hessian},
author={Watanabe,J},
journal={Proc. Sch. Sci. Tokai Univ.},
volume={49},
 pages={1--21},
year={2014}
 }

@book{WB,
    author = {de Bondt, M. and Watanabe, J.},
    title = {On the theory of Gordan-Noether on homogeneous forms with zero Hessian (improved version)},
    publisher = { Springer, Cham },
    year = {2020},
     series = { Springer Proc. Math. Stat., vol. 319, pp. 73--107}
     }
\Addresses
\end{document}